\documentclass[11pt]{amsart}
\usepackage[utf8]{inputenc}
\usepackage[T1]{fontenc}
\usepackage{amsmath,amssymb,amsthm}
\usepackage{graphicx}
\usepackage[margin=2.7cm]{geometry}
\usepackage{hyperref}

\newtheorem{theorem}{Theorem}
\newtheorem{proposition}[theorem]{Proposition}
\newtheorem{corollary}[theorem]{Corollary}

\theoremstyle{remark}
\newtheorem{remark}[theorem]{Remark}

\begin{document}

\title[Minimal translation hypersurfaces in hyperbolic space]{Minimal translation hypersurfaces in hyperbolic space}

\author{Ameth Ndiaye}
\address{D\'epartement de Math\'ematiques, FASTEF, Universit\'e Cheikh Anta Diop, Dakar, S\'en\'egal}
\email{ameth1.ndiaye@ucad.edu.sn}

\subjclass[2020]{53A10, 53C42}
\keywords{Hyperbolic space, minimal hypersurface, singular minimal hypersurface, translation hypersurface, $\alpha$-catenary, Scherk surface}

\begin{abstract}
In the upper half-space model of the hyperbolic space $\mathbb{H}^{n+1}$, a translation hypersurface of type I is a graph $x_{n+1}=f_1(x_1)+\cdots+f_n(x_n)$ and a translation hypersurface of type II is a graph $x_n=f_1(x_1)+\cdots+f_{n-1}(x_{n-1})+g(x_{n+1})$. We classify the minimal translation hypersurfaces of both types for every $n\geq 2$. We obtain these results as particular cases of a classification of the singular minimal translation hypersurfaces of $\mathbb{R}^{n+1}$ for an arbitrary exponent $\alpha\neq 0$, with respect to a direction orthogonal (type I) or parallel (type II) to the hyperplane over which the hypersurface is a graph. The minimal translation hypersurfaces of type I are exactly the parabolic cylinders over a one-parameter family of arches, which contradicts a nonexistence theorem in the literature. The minimal translation hypersurfaces of type II are the parabolic cylinders, the vertical totally geodesic hyperplanes and, when $n\geq 3$, the vertical cylinders over the Scherk surface.
\end{abstract}

\maketitle
\section{Introduction}

Let $\mathbb{R}^{n+1}_+=\{(x_1,\ldots,x_{n+1})\in\mathbb{R}^{n+1} : x_{n+1}>0\}$ be the upper half-space model of the hyperbolic space $\mathbb{H}^{n+1}$, with metric $(dx_1^2+\cdots+dx_{n+1}^2)/x_{n+1}^2$. The coordinates $x_1,\ldots,x_n$ play the same role while $x_{n+1}$ is distinguished, and this led L\'opez \cite{Lo11} for $n=2$, and Seo \cite{Se13} for $n\geq 2$, to consider two classes of translation hypersurfaces.

A hypersurface of $\mathbb{H}^{n+1}$ is a \emph{translation hypersurface of type I} if it is locally given by $x_{n+1}=f_1(x_1)+\cdots+f_n(x_n)$, and a \emph{translation hypersurface of type II} if it is locally given by $x_n=f_1(x_1)+\cdots+f_{n-1}(x_{n-1})+g(x_{n+1})$, where $f_1,\ldots,f_n,g$ are smooth functions of one variable.

Type I hypersurfaces are graphs over the horosphere $\{x_{n+1}=1\}$ and type II hypersurfaces are graphs over the totally geodesic hyperplane $\{x_n=0\}$. The question is to find the minimal ones, in analogy with the classical theorem of Scherk in $\mathbb{R}^3$ and its extension to $\mathbb{R}^{n+1}$ by Dillen, Verstraelen and Zafindratafa \cite{DVZ}.

Minimal translation surfaces in $\mathbb{H}^3$ were introduced and first studied by L\'opez \cite{Lo11}. For type II, following the method of \cite{Lo11}, Seo obtained a one-parameter family of minimal surfaces $y=ax+b+g(z)$ which are not totally geodesic \cite[Theorem 3.3]{Se13}. For type I, L\'opez classified the singular minimal surfaces $z=f(x)+g(y)$ with respect to the vertical direction: they are the $\alpha$-catenary cylinders \cite[Theorem 4]{Lo18}, the proof given for $\alpha=1$ in \cite{Lo16} extending directly to every $\alpha$. Since the minimal surfaces of $\mathbb{H}^3$ are the singular minimal surfaces with $\alpha=-2$ \cite[p.~522]{Lo18}, this result gives the minimal translation surfaces of type I in $\mathbb{H}^3$, namely the $(-2)$-catenary cylinders. Up to a rotation about a vertical axis, these cylinders also appear in the recent work of Ferreira and dos Santos \cite[Theorem 3.1]{FS} on translation surfaces for the Lie group structure of $\mathbb{H}^3$.

In higher dimension, Seo \cite[Theorem 3.2]{Se13} states that there is no minimal translation hypersurface of type I in $\mathbb{H}^{n+1}$; the case $n=2$ is referred to \cite{Lo11}, and the statement is quoted in \cite[Theorem 3]{LSSS}. For $n\geq 3$, the proof begins by showing that $f_i''\neq 0$ for every $i$: if $f_j''=0$, then $f_j(x_j)=ax_j+b$, and it is noted that $a\neq 0$ ``since $M$ is a graph''. The case $a=0$, where $f_j$ is constant, is not excluded by this remark, and it is precisely the case where minimal examples exist. As far as we know, minimal translation hypersurfaces of type II in $\mathbb{H}^{n+1}$, $n\geq 3$, have not been studied as such; however, as explained below, they are contained in the setting of \cite[Theorem 4.1]{EEA}.

Both problems are particular cases of a question on singular minimal hypersurfaces. Given $\alpha\in\mathbb{R}$ and a unit vector $\mathbf{a}\in\mathbb{R}^{n+1}$, a hypersurface $M$ of the half-space $\{p : \langle p,\mathbf{a}\rangle>0\}$ with unit normal $N$ is an \emph{$\alpha$-singular minimal hypersurface} with respect to $\mathbf{a}$ if its mean curvature $H$ (the trace of the shape operator divided by $n$) satisfies
\begin{equation}\label{eq-sm}
nH(p)=\alpha\,\frac{\langle N(p),\mathbf{a}\rangle}{\langle p,\mathbf{a}\rangle},\qquad p\in M.
\end{equation}
These hypersurfaces are the critical points of the potential energy $\int_M\langle p,\mathbf{a}\rangle^\alpha\,dM$ \cite{Di03}, and they appear in several contexts.

From the physical point of view, when $\mathbf{a}$ points upwards and $\alpha=1$, the energy is the gravitational potential energy of a homogeneous heavy surface, and \eqref{eq-sm} describes the surfaces of least potential energy with prescribed boundary. These surfaces are the two-dimensional analogue of the catenary, studied by B\"ohme, Hildebrandt and Tausch \cite{BHT}; for general $\alpha$ and in higher dimension, existence, regularity and nonexistence results were obtained by Bemelmans and Dierkes \cite{BD}, Dierkes and Huisken \cite{DH} and Dierkes \cite{Di03,Di19}. The case $\alpha=1$ is also related to the shape of cupolas and domes in architecture \cite{Lo21}.

From the geometric point of view, since the area element of a hypersurface for the conformal metric $\langle p,\mathbf{a}\rangle^{2\alpha/n}\langle\cdot,\cdot\rangle$ is $\langle p,\mathbf{a}\rangle^{\alpha}\,dM$, the $\alpha$-singular minimal hypersurfaces are exactly the minimal hypersurfaces of the half-space $\{\langle p,\mathbf{a}\rangle>0\}$ endowed with this conformally flat metric. Equivalently, they are the weighted minimal hypersurfaces of $\mathbb{R}^{n+1}$ for the density $\langle p,\mathbf{a}\rangle^{\alpha}$. The parameter $\alpha$ thus gives a one-parameter family of geometric problems containing the classical minimal hypersurfaces ($\alpha=0$) and, for $\mathbf{a}=e_{n+1}$ and $\alpha=-n$, the minimal hypersurfaces of the hyperbolic space, whose Dirichlet problem at infinity has been studied by Anderson \cite{An} and Lin \cite{Li}. Working with a general $\alpha$ makes clear which properties of the hyperbolic case depend only on the structure of equation \eqref{eq-sm}, and it gives the hyperbolic and the Euclidean weighted problems at once.

Since the hyperbolic mean curvature $H_h$ of a hypersurface of $\mathbb{R}^{n+1}_+$ is related to its Euclidean mean curvature $H$ by
\begin{equation}\label{eq-hyp}
H_h=x_{n+1}H+N_{n+1},
\end{equation}
a hypersurface is minimal in $\mathbb{H}^{n+1}$ if and only if it is $(-n)$-singular minimal with respect to $\mathbf{a}=e_{n+1}$ \cite[Remark 1.3]{Di19}, see also \cite[p.~522]{Lo18} for $n=2$. Type I hypersurfaces correspond to the case where $\mathbf{a}$ is orthogonal to the hyperplane over which $M$ is a graph, and type II hypersurfaces to the case where $\mathbf{a}$ is parallel to it. For $n=2$, both cases are solved in \cite[Theorems 4 and 5]{Lo18}. For $n\geq 3$, the first case is stated as an open problem by Erdur, Erg\"ut and Ayd\i n \cite[Section 6]{EEA}, and the second one is the subject of \cite[Theorem 4.1]{EEA}.

For $\alpha\neq 0$, an \emph{$\alpha$-catenary} is a planar curve which is the graph of a non-constant positive solution of
\begin{equation}\label{eq-cat}
h\,h''=\alpha\,(1+h'^2).
\end{equation}
Our first result gives the singular minimal translation hypersurfaces with respect to the vertical direction.

\begin{theorem}\label{th-A}
Let $n\geq 2$, $\alpha\neq 0$ and let $M$ be the graph of $x_{n+1}=f_1(x_1)+\cdots+f_n(x_n)>0$. Then $M$ is an $\alpha$-singular minimal hypersurface with respect to $e_{n+1}$ if and only if, after a permutation of the coordinates, $f_2,\ldots,f_n$ are constant and $h=f_1+\cdots+f_n$ is a solution of \eqref{eq-cat}. In other words, $M$ is an open subset of a cylinder over an $\alpha$-catenary whose rulings are orthogonal to $e_{n+1}$.
\end{theorem}

Our second result gives the singular minimal translation hypersurfaces with respect to a direction contained in the hyperplane over which they are graphs. We use the coordinates $(x_1,\ldots,x_{n-1},y,z)$ of $\mathbb{R}^{n+1}$, so that $\mathbf{a}=e_{n+1}$ is the direction of $z$.

\begin{theorem}\label{th-B}
Let $n\geq 2$, $\alpha\neq 0$ and let $M$ be the hypersurface $y=f_1(x_1)+\cdots+f_{n-1}(x_{n-1})+g(z)$, $z>0$. Then $M$ is an $\alpha$-singular minimal hypersurface with respect to $e_{n+1}$ if and only if one of the following holds:
\begin{enumerate}
\item $g$ is constant and $y=f_1(x_1)+\cdots+f_{n-1}(x_{n-1})$ is a minimal translation hypersurface of $\mathbb{R}^n$. Then $M$ is an open subset of $M_0\times\mathbb{R}e_{n+1}$, where $M_0\subset\{z=0\}\cong\mathbb{R}^n$ is a hyperplane or, when $n\geq 3$, a cylinder $\Sigma\times\mathbb{R}^{n-3}$ over a Scherk minimal translation surface $\Sigma$.
\item $f_i(x_i)=m_ix_i+d_i$ for $1\leq i\leq n-1$, and, with $K=1+\sum_i m_i^2$, the function $g$ is non-constant and satisfies
\[
\frac{g'^2}{K+g'^2}=\lambda\, z^{-2\alpha},\qquad\lambda>0.
\]
Then $M$ is an open subset of a cylinder over an $\alpha$-catenary whose rulings are orthogonal to $e_{n+1}$.
\end{enumerate}
\end{theorem}

For $n=2$, after a relabelling of the coordinates, Theorem \ref{th-B} is the classification of L\'opez for a horizontal vector \cite[Theorem 5]{Lo18}. For $n\geq 3$, it corresponds, after a relabelling of the coordinates, to \cite[Theorem 4.1]{EEA}, whose statement contains the hyperplanes and the $\alpha$-catenary cylinders. In the proof of \cite[Theorem 4.1]{EEA}, the case corresponding to $g$ constant is not considered, because the hypersurface is then minimal. This case contains the cylinders over the Scherk surface, which are $\alpha$-singular minimal for every $\alpha$ because $H=0$ and $\langle N,e_{n+1}\rangle=0$, and which should therefore be added to the statement. Our proof of Theorem \ref{th-B} is independent of \cite{EEA}.

Taking $\alpha=-n$, we obtain the classification in hyperbolic space. For $c>0$, let $h_c$ be the maximal solution of
\begin{equation}\label{eq-cyl}
h\,h''=-n\,(1+h'^2),\qquad h(0)=c^{1/(2n)},\quad h'(0)=0,
\end{equation}
and let $\mathcal{P}^n_c=\{(x_1,\ldots,x_n,h_c(x_n))\}\subset\mathbb{H}^{n+1}$. This hypersurface is invariant by the $(n-1)$-dimensional group of parabolic isometries $(x_1,\ldots,x_{n+1})\mapsto(x_1+t_1,\ldots,x_{n-1}+t_{n-1},x_n,x_{n+1})$, and we call it a \emph{minimal parabolic cylinder}.

\begin{corollary}\label{co-I}
A translation hypersurface of type I in $\mathbb{H}^{n+1}$, $n\geq 2$, is minimal if and only if it is, up to a permutation of $x_1,\ldots,x_n$ and a horizontal translation, an open subset of a minimal parabolic cylinder $\mathcal{P}^n_c$. For $n=2$, this is \cite[Theorem 4]{Lo18} with $\alpha=-2$, which completes \cite[Theorem 1.2]{Lo11}. For $n\geq 3$, it shows that the statement of \cite[Theorem 3.2]{Se13} has to be corrected.
\end{corollary}

\begin{corollary}\label{co-II}
A translation hypersurface of type II in $\mathbb{H}^{n+1}$, $n\geq 2$, is minimal if and only if it is an open subset of one of the following:
\begin{enumerate}
\item a vertical hyperplane, which is totally geodesic;
\item when $n\geq 3$, a vertical cylinder $\{(q,t) : q\in\Sigma\times\mathbb{R}^{n-3},\ t>0\}$, where $\Sigma$ is a Scherk minimal translation surface in a horizontal $3$-space; up to isometries and dilations, $\Sigma$ is the surface $x_3=\log(\cos x_2/\cos x_1)$;
\item the image of a minimal parabolic cylinder $\mathcal{P}^n_c$ by a rotation about a vertical axis and a horizontal translation.
\end{enumerate}
The hypersurfaces of item (2) are not totally geodesic and are not parabolic cylinders.
\end{corollary}

For $n=2$, Corollary \ref{co-II} agrees with \cite[Theorem 3.3]{Se13}. The minimal hypersurfaces of item (2) are instances of the following simple fact, which follows from \eqref{eq-hyp} since $N_{n+1}=0$: if $\Sigma$ is a minimal hypersurface of $\mathbb{R}^n$, then the vertical cylinder $\{(q,t) : q\in\Sigma,\ t>0\}$ is minimal in $\mathbb{H}^{n+1}$.

The paper is organized as follows. Section \ref{sec-pre} contains the equations and the properties of the $\alpha$-catenaries that we need. Theorems \ref{th-A} and \ref{th-B} are proved in Sections \ref{sec-A} and \ref{sec-B}, and the consequences in hyperbolic space are discussed in Section \ref{sec-hyp}.

\section{Preliminaries}\label{sec-pre}

Throughout the paper we use the following convention. The functions involved are real analytic: a singular minimal hypersurface satisfies an elliptic equation with analytic coefficients, so it is real analytic, and so are the functions $f_i$ and $g$, which are obtained by restricting its defining function to coordinate lines. Consequently, each of the functions $f_i'$, $f_i''$, $f_i'''$, $g'$, \ldots\ either vanishes identically or has isolated zeros, and all identities obtained on an open set extend to the whole domain. All the arguments below are local, and we work on an open set where the functions which do not vanish identically have no zeros.

We write $a_i=f_i'$, $b_i=f_i''$ and $W^2=1+\sum_i a_i^2$ for a graph $x_{n+1}=F=f_1(x_1)+\cdots+f_n(x_n)$ with upward unit normal $N=(-a_1,\ldots,-a_n,1)/W$. Its mean curvature is \cite{Se13}
\[
nH=\frac{1}{W^3}\sum_{i=1}^n\bigl(W^2-a_i^2\bigr)b_i,
\]
and \eqref{eq-sm} with $\mathbf{a}=e_{n+1}$ becomes
\begin{equation}\label{eq-I}
F\,B=\alpha\,W^2,\qquad B=\sum_{i=1}^n\bigl(W^2-a_i^2\bigr)b_i .
\end{equation}

For a hypersurface $y=\sum_{i=1}^{n-1}f_i(x_i)+g(z)$ as in Theorem \ref{th-B}, we write $p=g'$, $q=g''$ and
\[
K=1+\sum_{i=1}^{n-1}a_i^2,\qquad S=\sum_{i=1}^{n-1}b_i,\qquad T=\sum_{i=1}^{n-1}a_i^2b_i .
\]
This is a graph over the hyperplane $\{y=0\}$, with unit normal $N=(-a_1,\ldots,-a_{n-1},1,-p)/W$ and $W^2=K+p^2$. The same formula for the mean curvature gives $nHW^3=W^2S-T+Kq$, and \eqref{eq-sm} with $\langle N,e_{n+1}\rangle=-p/W$ and $\langle p,e_{n+1}\rangle=z$ becomes
\begin{equation}\label{eq-II}
z\bigl[(K+p^2)S-T+Kq\bigr]=-\alpha\,p\,(K+p^2).
\end{equation}
For $\alpha=-n$, \eqref{eq-I} and \eqref{eq-II} are the equations $H_h=0$ obtained from \eqref{eq-hyp}; we have checked both with a computer algebra system.

\begin{proposition}\label{pr-cat}
Let $\alpha\neq 0$. Every non-constant positive solution of \eqref{eq-cat} satisfies $1+h'^2=Ch^{2\alpha}$ for a constant $C>0$. If $\alpha<0$, every maximal solution is, up to a translation of the variable, even and concave, defined on a bounded interval $(-r,r)$, and satisfies $h\to 0$, $|h'|\to\infty$ at $\pm r$. For $\alpha=-n$ and $h(0)=c^{1/(2n)}$, $h'(0)=0$, we have
\[
h^{2n}(1+h'^2)=c,\qquad r=r_c=c^{1/(2n)}\int_0^1\frac{t^n\,dt}{\sqrt{1-t^{2n}}}=\frac{c^{1/(2n)}}{2n}\,B\Bigl(\frac{n+1}{2n},\frac12\Bigr).
\]
\end{proposition}

\begin{proof}
Where $h'\neq 0$, \eqref{eq-cat} gives $\bigl(\log(1+h'^2)\bigr)'=2h'h''/(1+h'^2)=2\alpha h'/h$, hence $1+h'^2=Ch^{2\alpha}$; since $h$ is not constant, the zeros of $h'$ are isolated and the identity holds everywhere by continuity. The qualitative properties for $\alpha<0$ are \cite[Theorem 3]{Lo18}. For $\alpha=-n$, the maximum of $h$ is $h(0)=C^{-1/(2n)}$, so $C=1/c$. On $(0,r)$ we have $h'<0$ and $y=\int_{h(y)}^{c^{1/(2n)}}s^n\,ds/\sqrt{c-s^{2n}}$; the substitutions $s=c^{1/(2n)}t$ and $u=t^{2n}$ give the value of $r_c$.
\end{proof}

For $n=2$ we recover $r_c\approx 0.5991\,c^{1/4}$. Figure \ref{fig-prof} shows the profiles for $n=2,3,4$.

\begin{figure}[ht]
\centering
\includegraphics[width=0.7\textwidth]{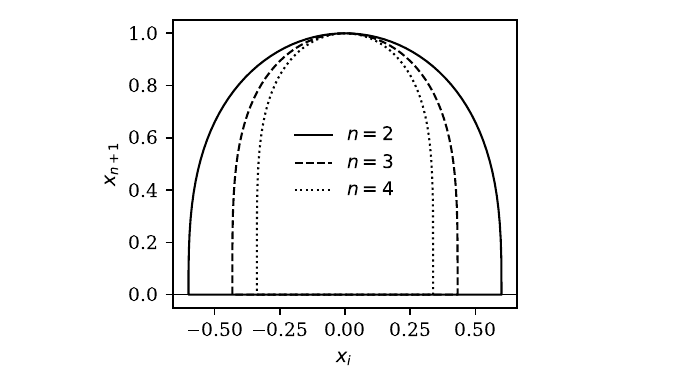}
\caption{Profiles of the minimal parabolic cylinders $\mathcal{P}^n_1$ of $\mathbb{H}^{n+1}$, solutions of $h^{2n}(1+h'^2)=1$, for $n=2,3,4$.}
\label{fig-prof}
\end{figure}

\section{Proof of Theorem \ref{th-A}}\label{sec-A}

Suppose that $M$ satisfies \eqref{eq-I}. Let $J$ be the set of indices $i$ such that $f_i$ is not constant, and $k$ its cardinality. The functions $f_i$ with $i\notin J$ only contribute a constant to $F$ and do not appear in $W$ and $B$.

\emph{Step 1: $k\geq 1$.} If $k=0$, then $F$ is constant, $B=0$ and \eqref{eq-I} gives $\alpha=0$.

\emph{Step 2: no $f_i$ is affine and non-constant.} Suppose that $f_j(x_j)=ax_j+b$ with $a\neq 0$. Then $b_j=0$ and $a_j=a$, so $B$ and $W$ do not depend on $x_j$, while $F$ is affine in $x_j$ with leading coefficient $a$. Comparing the coefficients of $x_j$ in \eqref{eq-I} gives $aB=0$, hence $B=0$ and $\alpha W^2=0$, a contradiction. Consequently $b_i\neq 0$ for $i\in J$.

\emph{Step 3: if $k\geq 2$, then $b_i'\neq 0$ for every $i\in J$.} Suppose that $b_i=\beta_i$ is a non-zero constant for some $i\in J$. Then $a_i$ and $f_i$ are polynomials of degree $1$ and $2$ in $x_i$. Writing $A=1+\sum_{l\neq i}a_l^2$, we have $W^2=A+a_i^2$ and
\[
B=A\beta_i+\sum_{l\neq i}(A-a_l^2)b_l+a_i^2\sigma_i,\qquad \sigma_i=\sum_{l\in J,\,l\neq i}b_l .
\]
As a function of $x_i$, the left-hand side of \eqref{eq-I} is a polynomial whose coefficient of $x_i^4$ is $\beta_i^3\sigma_i/2$, and the right-hand side is a polynomial of degree $2$. Therefore $\sigma_i=0$, and differentiating $\sigma_i$ with respect to $x_l$ gives $b_l'=0$ for every $l\in J\setminus\{i\}$. Thus every $b_l$, $l\in J$, is a non-zero constant $\beta_l$, and the same argument gives $\sum_{l\in J,\,l\neq i}\beta_l=0$ for every $i\in J$. If $k=2$, this says that the other constant vanishes, a contradiction. If $k\geq 3$, summing over $i\in J$ gives $(k-1)\sum_{l\in J}\beta_l=0$, hence $\beta_i=\sum_{l\in J}\beta_l-\sigma_i=0$ for every $i$, again a contradiction.

\emph{Step 4: $k\leq 2$.} Suppose $k\geq 3$ and let $1,2,3\in J$. We differentiate \eqref{eq-I} with respect to $x_1$, $x_2$ and $x_3$. Since $\partial_i\partial_jF=0$ and $W^2-a_l^2$ is a sum of functions of one variable not depending on $x_l$, we have $\partial_1\partial_2\partial_3(W^2)=0$, $\partial_1\partial_2\partial_3B=0$ and $\partial_2\partial_3B=2(a_3b_3b_2'+a_2b_2b_3')$. We obtain
\[
a_1\bigl(a_2b_2b_3'+a_3b_3b_2'\bigr)+a_2\bigl(a_3b_3b_1'+a_1b_1b_3'\bigr)+a_3\bigl(a_1b_1b_2'+a_2b_2b_1'\bigr)=0 .
\]
By Steps 2 and 3, $a_ib_ib_i'\neq 0$. With $u_i=b_i'/(a_ib_i)$ and $v_i=1/b_i$, the division by $a_1b_1a_2b_2a_3b_3$ gives
\begin{equation}\label{eq-uv}
v_1(u_2+u_3)+v_2(u_3+u_1)+v_3(u_1+u_2)=0 .
\end{equation}
Note that $v_i'=-b_i'/b_i^2\neq 0$, and that $v_2+v_3$ is not identically zero, since otherwise $v_2$ and $v_3$ would be constant. Differentiating \eqref{eq-uv} with respect to $x_1$ gives $v_1'(u_2+u_3)+u_1'(v_2+v_3)=0$, that is,
\[
\frac{u_1'}{v_1'}=-\frac{u_2+u_3}{v_2+v_3}.
\]
The left-hand side depends only on $x_1$ and the right-hand side only on $(x_2,x_3)$, so both are equal to a constant $\lambda_1$. Hence $u_1=\lambda_1v_1+\mu_1$ and, in the same way, $u_i=\lambda_iv_i+\mu_i$ for $i=2,3$. Plugging into $u_2+u_3=-\lambda_1(v_2+v_3)$ gives
\[
(\lambda_2+\lambda_1)v_2+(\lambda_3+\lambda_1)v_3+\mu_2+\mu_3=0,
\]
and since $v_2$ and $v_3$ are non-constant functions of different variables, $\lambda_2=\lambda_3=-\lambda_1$ and $\mu_2=-\mu_3$. The same argument with the roles of the indices exchanged gives $\lambda_i=-\lambda_j$ and $\mu_i=-\mu_j$ for all $i\neq j$ in $\{1,2,3\}$. Hence $\lambda_1=-\lambda_2=\lambda_3=-\lambda_1$ and similarly for $\mu$, so all $\lambda_i$ and $\mu_i$ vanish. Then $u_i=0$, that is, $b_i'=0$, contradicting Step 3.

\emph{Step 5: $k\neq 2$.} Suppose $J=\{1,2\}$ and write $f=f_1$, $g=f_2$, $c=\sum_{i\geq 3}f_i$. Equation \eqref{eq-I} reads $(c+f+g)\bigl[(1+g'^2)f''+(1+f'^2)g''\bigr]=\alpha(1+f'^2+g'^2)$. Let
\[
A_1=\frac{1}{1+f'^2},\quad A_2=\frac{1}{1+g'^2},\quad \varphi=\frac{f''}{1+f'^2},\quad \psi=\frac{g''}{1+g'^2}.
\]
Multiplying the equation by $A_1A_2$ gives
\[
(c+f+g)(\varphi+\psi)=\alpha\,(A_1+A_2-A_1A_2).
\]
Since $A_1'=-2f'f''A_1^2$ and $A_2'=-2g'g''A_2^2$, differentiation with respect to $x_1$ and $x_2$ gives $f'\psi'+g'\varphi'=-\alpha A_1'A_2'=-4\alpha f'g'(\varphi A_1)(\psi A_2)$, and after division by $f'g'$,
\[
\frac{\varphi'}{f'}+\frac{\psi'}{g'}=-4\alpha\,(\varphi A_1)(\psi A_2).
\]
Differentiating with respect to $x_1$ and $x_2$ gives $(\varphi A_1)'(\psi A_2)'=0$. Suppose, for instance, that $\varphi A_1=f''/(1+f'^2)^2$ is a constant $\kappa$; then $\kappa\neq 0$ by Step 2. The last identity shows that $\varphi'/f'$ depends only on $x_2$, hence is a constant. Since $\varphi=\kappa(1+f'^2)$, we have $\varphi'/f'=2\kappa f''$, so $f''$ is constant, and then $f''=\kappa(1+f'^2)^2$ forces $f'$ to be constant, that is, $f''=0$, a contradiction.

\emph{Step 6: conclusion.} We have $k=1$, say $J=\{1\}$, and \eqref{eq-I} reduces to $h\,h''=\alpha(1+h'^2)$ for $h=F$, which is \eqref{eq-cat}. Conversely, if $F=h(x_1)$ with $h$ a solution of \eqref{eq-cat}, then $W^2=1+h'^2$, $B=h''$ and \eqref{eq-I} holds. This proves Theorem \ref{th-A}. \qed

\begin{remark}
For $n=2$, Theorem \ref{th-A} is \cite[Theorem 4]{Lo18}, for which only a reference to the method of \cite{Lo16} was given. Step 2 completes the claim at the beginning of the proof of \cite[Theorem 3.2]{Se13}, Step 3 corresponds to its Cases (1) and (2), and Step 4 is a simplified version of its Case (3).
\end{remark}

\section{Proof of Theorem \ref{th-B}}\label{sec-B}

We first check that the hypersurfaces of Theorem \ref{th-B} are $\alpha$-singular minimal. In case (1), $p=0$ and \eqref{eq-II} reduces to $KS-T=\sum_i(K-a_i^2)b_i=0$, which is the equation of minimal translation hypersurfaces of $\mathbb{R}^n$. In case (2), $S=T=0$, $K$ is constant and \eqref{eq-II} reads
\begin{equation}\label{eq-II2}
zKq=-\alpha p(K+p^2).
\end{equation}
Where $p\neq 0$, this is equivalent to $\bigl(\log\frac{p^2}{K+p^2}\bigr)'=\frac{2Kq}{p(K+p^2)}=-\frac{2\alpha}{z}$, whose solutions are $p^2/(K+p^2)=\lambda z^{-2\alpha}$. To identify the hypersurface, let $\nu=(-m_1,\ldots,-m_{n-1},1,0)/\sqrt{K}$ and $s=\langle(x,y,z),\nu\rangle$, so that $M$ is the cylinder with rulings orthogonal to $\nu$ and $e_{n+1}$ over the curve $s=\ell(z)=\bigl(g(z)+\sum d_i\bigr)/\sqrt{K}$ of the $(s,z)$-plane. Then $\ell'^2=\lambda z^{-2\alpha}/(1-\lambda z^{-2\alpha})$, and the inverse function $z=h(s)$ satisfies $1+h'^2=\lambda^{-1}h^{2\alpha}$. By Proposition \ref{pr-cat}, $h$ is an $\alpha$-catenary. In case (1), the classification follows from the theorem of Dillen, Verstraelen and Zafindratafa \cite{DVZ}, see also \cite[Theorem 2.1]{Se13}: a minimal translation hypersurface of $\mathbb{R}^n$ is a hyperplane or, if $n\geq 3$, a cylinder over a Scherk minimal translation surface.

Conversely, suppose that $M$ satisfies \eqref{eq-II}. If $p=0$ we are in case (1). From now on $p\neq 0$. Dividing \eqref{eq-II} by $z$, we write it as
\begin{equation}\label{eq-sep}
U+S\,P+K\,R+Q=0,\qquad U=SK-T,
\end{equation}
where $U,S,K$ depend only on $x=(x_1,\ldots,x_{n-1})$, and
\[
P=p^2,\qquad R=q+\frac{\alpha p}{z},\qquad Q=\frac{\alpha p^3}{z}
\]
depend only on $z$. Differentiating \eqref{eq-sep} with respect to $x_i$ gives
\begin{equation}\label{eq-sepi}
U_i+S_iP+K_iR=0,\qquad 1\leq i\leq n-1,
\end{equation}
where the subscript denotes the partial derivative with respect to $x_i$.

\emph{Case 1: the functions $1,P,R$ are linearly independent.} Then \eqref{eq-sepi} gives $K_i=2a_ib_i=0$ for every $i$, hence every $f_i$ is affine. Then $S=T=0$, $K$ is constant, and \eqref{eq-sep} is \eqref{eq-II2}, so we are in case (2).

\emph{Case 2: $P$ is constant.} Then $p$ is a non-zero constant, $q=0$, and \eqref{eq-sep} reads $U+SP+\alpha p(K+p^2)/z=0$. Differentiating with respect to $z$ gives $\alpha p(K+p^2)=0$, which is impossible.

\emph{Case 3: $P$ is not constant and $1,P,R$ are linearly dependent.} Then $R=mP+r$ for some constants $m,r$, and \eqref{eq-sep} becomes
\[
(U+rK)+(S+mK)P+Q=0 .
\]
Differentiating with respect to $x_i$ and using that $P$ is not constant, we get that $U+rK=c_1$ and $S+mK=c_2$ are constants, and then $c_1+c_2P+Q=0$. With $\nu=-\alpha\neq 0$, the conditions on $p$ are
\begin{equation}\label{eq-sys}
C(p,z):=\nu p^3-c_2zp^2-c_1z=0,\qquad q=mp^2+r+\frac{\nu p}{z}.
\end{equation}

Suppose first that $c_1=0$. Then $p^2(\nu p-c_2z)=0$ and, since $p\neq 0$, $p=c_2z/\nu$ with $c_2\neq 0$. The second equation of \eqref{eq-sys} gives $c_2/\nu=mc_2^2z^2/\nu^2+r+c_2$, so $m=0$ and $r=c_2/\nu-c_2$. We return to the horizontal variables. The identity $S=c_2$ shows that each $b_i$ is a constant $\beta_i$, with $\sum_i\beta_i=c_2$. The identity $U+rK=0$ reads $T=(c_2+r)K$, that is,
\[
\sum_{i=1}^{n-1}\beta_ia_i^2=\frac{c_2}{\nu}\Bigl(1+\sum_{i=1}^{n-1}a_i^2\Bigr).
\]
For the indices with $\beta_i\neq 0$, $a_i$ is a non-constant affine function of $x_i$, and for the other indices $a_i$ is constant. The identity is thus a polynomial identity in the independent variables $a_i$, $\beta_i\neq 0$, and its constant term gives $0=(c_2/\nu)\bigl(1+\sum_{\beta_i=0}a_i^2\bigr)$, a contradiction.

Suppose now that $c_1\neq 0$. If $C_p=3\nu p^2-2c_2zp$ vanishes identically, then $p=2c_2z/(3\nu)$ and $C=0$ gives $-c_2zp^2/3=c_1z$, so that $P$ is constant, which is excluded. Hence $C_p\neq 0$, and differentiating $C(p(z),z)=0$ gives $p'=(c_2p^2+c_1)/C_p$. Comparing with the second equation of \eqref{eq-sys}, we see that $p(z)$ is, for every $z$, a common root of $C(\cdot,z)$ and of
\[
D(p,z)=(c_2p^2+c_1)z-\bigl((mp^2+r)z+\nu p\bigr)\bigl(3\nu p^2-2c_2zp\bigr).
\]
Therefore the resultant $\mathrm{Res}_p(C,D)$, which is a polynomial in $z$, vanishes identically. A direct computation shows that its lowest order term is $-c_1^3\nu^4(3\nu-1)^3z^3$. If $\nu\neq 1/3$, this contradicts $c_1\neq 0$. If $\nu=1/3$, the vanishing of the other coefficients gives successively $r=2c_2/3$, $9c_1m=4c_2^2$ and $c_2^3m(3c_1m-2c_2^2)^2=0$, which force $c_2=m=r=0$. Then $S=c_2=0$, so the $b_i$ are constants $\beta_i$ with $\sum\beta_i=0$, and $U=c_1$ gives $\sum_i\beta_ia_i^2=-c_1$. As above, the coefficient of $a_i^2$ gives $\beta_i=0$ for every $i$, hence $c_1=0$, a contradiction. This concludes the proof of Theorem \ref{th-B}. \qed

\begin{remark}
The resultant computation, as well as the other symbolic identities used in the proofs, has been checked with the computer algebra system SymPy.
\end{remark}

\section{Minimal translation hypersurfaces in hyperbolic space}\label{sec-hyp}

Corollaries \ref{co-I} and \ref{co-II} follow from Theorems \ref{th-A} and \ref{th-B} with $\alpha=-n$, and from Proposition \ref{pr-cat}. We add some remarks on the examples.

The minimal parabolic cylinder $\mathcal{P}^n_c$ is the orbit of the profile curve $\{(0,\ldots,0,s,h_c(s)) : |s|<r_c\}$ under the parabolic translations in the directions $x_1,\ldots,x_{n-1}$. It is properly embedded, and its profile curve has infinite hyperbolic length at both ends, since its hyperbolic arc length element $\sqrt{1+h_c'^2}\,dh/(h_c|h_c'|)$ behaves like $dh/h$ as $h\to 0$. Its asymptotic boundary consists of the two parallel hyperplanes $\{x_n=\pm r_c,\ x_{n+1}=0\}$ of $\partial_\infty\mathbb{H}^{n+1}$ together with the point at infinity. The hyperbolic principal curvatures of $\mathcal{P}^n_c$ are $x_{n+1}\kappa_i+N_{n+1}$, where $\kappa_i$ are the Euclidean ones; along the $n-1$ ruling directions they equal $N_{n+1}=1/\sqrt{1+h_c'^2}\neq 0$, so $\mathcal{P}^n_c$ is not totally geodesic. For $n=2$, $\mathcal{P}^2_c$ is the $(-2)$-catenary cylinder of \cite{Lo18} and, up to a rotation about a vertical axis, the minimal translation cylinder of \cite[Theorem 3.1]{FS}.

The hypersurfaces of Corollary \ref{co-II}, item (2), are the vertical cylinders over $\Sigma\times\mathbb{R}^{n-3}$, where $\Sigma$ is a Scherk surface. For $n=3$ and up to dilations and isometries, this is the hypersurface of $\mathbb{H}^4$
\[
x_3=\log\frac{\cos x_2}{\cos x_1},\qquad |x_1|,|x_2|<\frac{\pi}{2},\ x_4>0,
\]
which is a translation hypersurface of type II with $g$ constant; its horizontal sections are the Scherk surface of Figure \ref{fig-scherk}. Since $N_4=0$, its hyperbolic principal curvatures are $x_4\kappa_1$, $x_4\kappa_2$ and $0$, where $\kappa_1=-\kappa_2$ are the Euclidean principal curvatures of the Scherk surface, which do not vanish identically. It is therefore not totally geodesic. It is not congruent to a parabolic cylinder either: the principal curvatures of $\mathcal{P}^3_c$ are $N_4,N_4,-2N_4$ with $N_4=1/\sqrt{1+h_c'^2}\neq 0$, so none of them vanishes, whereas one principal curvature of the Scherk cylinder is identically zero.

\begin{figure}[ht]
\centering
\includegraphics[width=0.5\textwidth]{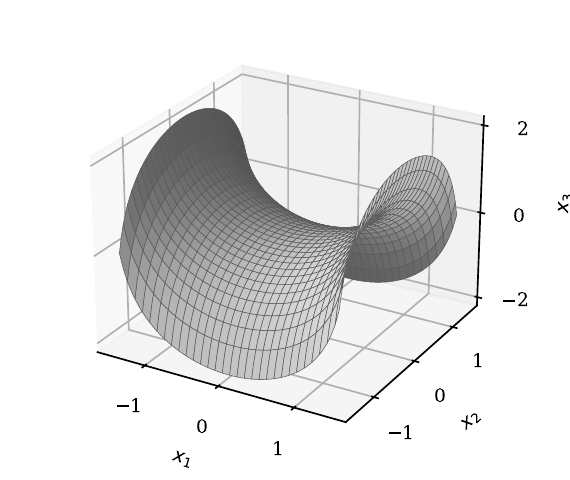}
\caption{A horizontal section $\{x_4=\mathrm{const}\}$ of the minimal hypersurface $x_3=\log(\cos x_2/\cos x_1)$ of $\mathbb{H}^4$.}
\label{fig-scherk}
\end{figure}

We conclude with some comments on the literature. For $n=2$, the classification of Corollaries \ref{co-I} and \ref{co-II} is contained in \cite[Theorem 3.3]{Se13} for type II and in \cite[Theorem 4]{Lo18} for type I, the latter being stated in terms of singular minimal surfaces. For $n\geq 3$, Corollary \ref{co-I} shows that \cite[Theorem 3.2]{Se13}, which is also quoted in \cite[Theorem 3]{LSSS}, does not hold as stated. Corollary \ref{co-II} does not seem to have been stated before for $n\geq 3$; it can also be derived from \cite[Theorem 4.1]{EEA} with $\alpha=-n$, once the case where $g$ is constant is added.

\end{document}